\documentclass[11pt,draft]{article}

\usepackage[margin=1in]{geometry}

\usepackage{placeins}
\usepackage{latexsym}
\usepackage{amsmath}
\usepackage{amsfonts}
\usepackage{amssymb}
\usepackage{enumerate}
\usepackage{cite}

\usepackage[latin1]{inputenc}

\FloatBarrier
\numberwithin{equation}{section}

\newtheorem{theorem}{Theorem}[section]
\newtheorem{lemma}[theorem]{Lemma}
\newtheorem{proposition}[theorem]{Proposition}
\newtheorem{corollary}[theorem]{Corollary}
\newtheorem{definition}[theorem]{Definition}

\newtheorem{remark}[theorem]{Remark}

\newcommand{\eproof}{\noindent\mbox{\framebox [0.8ex]{}}  \medskip}

\newcommand{\N}{\mathbf{N}}

\newcommand{\C}{\mathbf{C}}

\newcommand{\GL}{\mathbf{GL}}

\newcommand{\Zm}{\mathbf{Z}_m}
\newcommand{\Zthree}{\mathbf{Z}_3}

\newcommand{\IM}{\mathrm{Im}}
\newcommand{\dd}{\mathrm{d}}

\newcommand{\Psigmaj}{\mathcal{P}_{\sigma^j}}

\newcommand{\vecPsigmaj}{\vec{\mathcal{P}}_{\sigma^j}}

\newcommand{\vecP}{\vec{\mathcal{P}}}

\newcommand{\J}{\{0, \ldots, m-1 \}}
\newcommand{\Jo}{\{1, \ldots, m-1 \}}

\begin{document}

\title{Relative equivariants under compact Lie groups}

\author{Patr\'{\i}cia~H.~Baptistelli\\
{\small Departamento de Matem\'atica  - Universidade Estadual de Maring\'a}\\
{\small Av. Colombo, 5790, 87020-900 Maring\'a  PR, Brazil \footnote{Email address: phbaptistelli@uem.br}}
\and Miriam Manoel\\
{\small Departamento de Matem\'atica, ICMC - Universidade de S\~ao Paulo}\\
{\small C.P. 668, 13560-970 S\~ao Carlos SP, Brazil \footnote{Email address: miriam@icmc.usp.br (corresponding author)  }}
}
\date{}
\maketitle

\begin{abstract}

In this work we obtain the general form of  polynomial mappings that commute with a linear action of a relative symmetry group.  The aim is to give results for relative equivariant polynomials  that correspond to the results for relative invariants obtained in a previous paper [P.H. Baptistelli, M. Manoel (2013) Invariants and relative invariants under compact Lie groups, {\it J. Pure Appl. Algebra} 217, 2213--2220].  We present an algorithm to compute generators for relative equivariant submodules
from the invariant theory applied to the subgroup formed only by the symmetries. The same method provides, as a particular case, generators for equivariants under the whole group from the knowledge of equivariant generators by a smaller subgroup, which is normal of finite index. \\
\end{abstract}

\begin{tabular}{ll}
 {\it Keywords}: &   relative symmetry, invariant theory,  equivariant, finite index,
Molien formulae.
\end{tabular}

\vspace*{0.5cm}

\hspace*{0.1cm} {\it 2000 MSC}:  13A50; 34C14
\vspace*{2mm}

\section{Introduction} \label{sec:INTRO}

Occurrence of symmetries provides a powerful and useful tool in the formulation and analysis of  a model.  The appropriate combination of the geometry of a configuration space  with its inherent symmetries can show, in a natural way, existence of many phenomena not expected when symmetries are not present. This is due to the fact that  they are the reason for common observations such as degeneracy of solutions,  high codimension bifurcations, unexpected stabilities, phase relations, synchrony, periodicity of solutions, among many others. The set of  symmetries appears in the model with different algebraic structures in different contexts.  In particular, the study of dynamical systems has been greatly developed throughout many years when the collection of symmetries has a group structure. Group representation theory is in this case  a systematic way that enables us to deal with this question in distinct aspects in applications. More specifically,  in the numerical analysis of such systems  through computational programming as, for example, in \cite{Gat96,Gat, sturm}; also, in algebraic methods that provide the general form of the vector fields through the description of invariant functions and commuting mappings under distinct actions of symmetry groups as in \cite{antoneli1, antoneli, BM4}. In addition,  also based on group representation are the analysis of steady-state bifurcation, Hopf bifurcation, periodic cycles, classification of singularities and so on, for which we cite for example \cite{BM1, BM3, Buono, Lamb99} and the books \cite{chossat, GS69} with hundreds of references therein. More recently, in a related but distinct direction, many authors have investigated the occurrence of symmetries in networks of dynamical systems. Networks are schematically represented by a graph, and symmetries enter in a more subtle way, even when the architecture of the graph presents no  symmetry. The algebraic structure  in this case is of a groupoid of symmetries. This has been first formalized in \cite{SGP} and has been successfully applied since then by many authors.

The results in the present work  are related to symmetries in differential operations on complex variables.  We obtain a procedure from invariant theory to give the general form of mappings that commute with a complex linear action of a group $\Gamma$, the group of relative symmetries. We assume   that $\Gamma$  is a compact Lie group whose subgroup of symmetries is a normal subgroup $K$ of finite index $m$ greater than two.  More specifically, we consider a group epimorphism
\begin{equation} \label{eq:HOMOSIGMA}
 \sigma : \Gamma \to \Zm,
\end{equation}
where $\Zm$ denotes the cyclic group generated by the $m$-th root of unity and $K =  \ker \sigma$.
The case $m=2$ corresponds to reversible equivariant mappings and it has been studied in \cite{antoneli}. It models vector fields whose integral curves come into two types of symmetric families:  the time-preserving symmetric family, each element of which  given from an element of $K$ (the subgroup of symmetries) and the time-reversing symmetric family, whose elements come from each element of the complement $ \Gamma \setminus K$ (the subset of reversible symmetries).
Here we treat the case $m > 2$. Potential applications of this case are related to the theory of relativity in complex time, which not only complies with the possibility of time decomposition into two dimensions, but also conciliates with the idea of a complex space (see \cite{Naschie, MSSF}).  \\

From now on  we assume throughout that $\Gamma$ is a compact Lie group that admits an index-$m$ normal subgroup $K$, the kernel of (\ref{eq:HOMOSIGMA}), and  $\delta \in \Gamma$ is such that $\sigma(\delta) $ is the primitive $m$-th root of unity. We then have the decomposition of $\Gamma$ as a disjoint union of  left-cosets:
\begin{equation}
\label{eq:left-cosets}
\Gamma =  \ {\dot{\bigcup}}_{j=0}^{m-1}  \ \delta^{j} K.
\end{equation}
The results in \cite{BM4} regard relative invariants. In this paper, we derive the corresponding algebraic results for relative equivariants. \\

For a linear action of $\Gamma$ on a finite-dimensional vector space $V$ over ${\bf C}$,  consider the ring ${\cal P}(\Gamma)$ of the $\Gamma$-invariant polynomials, namely, the ring of polynomial functions $f : V \to \C$ such that
$f(\gamma x) = f(x),$ for all $\gamma \in \Gamma$ and $x \in V.$
For each $j \in \Jo,$ a polynomial function $f: V \to \C$ is called
{\it $\sigma^j$-relative invariant} if
\[ f(\gamma x)~=~\sigma^j(\gamma) f(x), \ \ \  \forall \gamma \in \Gamma, \  x \in V.\]
The set of the $\sigma^j$-relative invariant polynomial functions is a finitely generated graded module over
$\mathcal{P}(\Gamma)$ denoted by $\Psigmaj(\Gamma)$. In \cite[Theorem 2.4]{BM4} we have obtained the decomposition into modules
\begin{equation} \label{eq:DECOMPOSITION3}
 \mathcal{P}(K)~=~   \bigoplus_{j=0}^{m-1} \mathcal{P}_{\sigma^{j}}(\Gamma),
\end{equation}
which is the basis for the construction of a Hilbert basis for the ring $\mathcal{P}(\Gamma)$ ($j=0$) from the knowledge of a
Hilbert basis of the invariants only by the subgroup $K$ of  $\Gamma$. \\


In this paper we  find generators for  the modules $\mathcal{P}_{\sigma^j}(\Gamma)$, $j=1, \ldots, m-1$, from the $K$-invariants.  Also, we extend the algebraic formalism  of  \cite{antoneli} and apply results of \cite{BM4}  for the systematic study of relative equivariants.  In Section \ref{sec:RELATIVEEQUIV}, using  Reynolds operators on the module of the $K$-equivariants, we obtain our first main result,  Theorem~\ref{thm:RELATIVEDECOMP}, as an extension of the decomposition (\ref{eq:DECOMPOSITION3}) to the relative equivariant polynomial mappings. In Section~\ref{sec:MOLIEN} we register the Molien formulae associated to the Hilbert-Poincar\'e series of the dimensions of the spaces of homogeneous relative invariants and relative equivariants. The derivation of the formulae is a direct process from the results in \cite{Gat96} for equivariants; however, as a simple example shall reveal, a  subtlety in the general case makes this inclusion worthwhile.   This shall be used as a running example to illustrate and compare the results.
Section \ref{sec:GERADORES} presents the method to obtain generators for  $\mathcal{P}_{\sigma^j}(\Gamma)$, $j=1, \ldots, m-1$, Theorem~\ref{thm:RELATIVEGENERATORS}, which is also crucial for  our second main result, Theorem~\ref{thm:MAIN2}. This provides an algebraic way to obtain generators  for each of the $m$ modules of the relative equivariants over the ring $\mathcal{P}(\Gamma)$ from the knowledge of the generators for the $K$-invariants and $K$-equivariants.

\section{The relative equivariants} \label{sec:RELATIVEEQUIV}

For the next definition, we consider linear actions of $\Gamma$ on two
finite-dimensional vector spaces $V$ and $W$ over $\C.$ These actions correspond to
representations $\rho: \Gamma \to \GL(V)$ and $\eta: \Gamma \to \GL(W),$ where $\GL(V)$ is the group of
invertible linear transformations on $V$. We denote by $(\rho,V)$ and $(\eta, W)$
the vector spaces $V$ and $W$ endowed with each of the representations. 

\begin{definition} \label{def:equivariant}
A polynomial mapping $g: (\rho,V) \to (\eta,W) $ is called
$\Gamma$-equivariant if
\begin{equation} \label{eq:EQUIVARIANT}
g(\rho(\gamma) x)~=~\eta(\gamma) g(x)~, \quad \forall \gamma \in \Gamma, \ x \in V.
\end{equation}
\end{definition}

\quad

In the interest of generalizing results in previous papers  we consider here mappings with same vector space on the source and target, although it is a trivial matter to see that they hold equally well if these are distinct. So we consider henceforth two linear actions of  $\Gamma$  on $V,$ with representations $(\rho,V)$ and $(\eta, V).$ We shall also consider  the epimorphism $\sigma : \Gamma \to \Zm$ in (\ref{eq:HOMOSIGMA}) which defines the one-dimensional representation of $\Gamma$ on $\C$ given by $z \mapsto \sigma(\gamma)z,$ for $z \in \C$.
We now give our main definition:
\begin{definition} \label{def:relative equivariant}
For each $j \in \Jo,$ a polynomial mapping $g: (\rho,V) \to (\eta,V) $ is called
$\sigma^j$-relative equivariant if
\begin{equation} \label{eq:SIGMARELATIVEEQUIV}
g(\rho(\gamma) x)~=~\sigma^j(\gamma) \eta(\gamma) g(x)~, \quad \forall \gamma \in \Gamma, \ x \in V.
\end{equation}
\end{definition}

\quad

With this definition and  the decomposition (\ref{eq:left-cosets}) in mind, we shall call  elements in $K = \ker \sigma$ the {\it symmetries} of $\Gamma$ and elements of $\delta^j K$ the {\it $\sigma^j$-relative symmetries} of $\Gamma,$ for  $j \in \{1, \ldots, m-1\}.$ Note that if $g$ satisfies (\ref{eq:SIGMARELATIVEEQUIV}) for $j=0,$ then $g$ is a $\Gamma$-equivariant mapping on $V.$ The particular case
when $m=2$ and $\rho =  \eta$ corresponds to the reversible equivariant context of \cite{antoneli}.     \\

The set of the $\sigma^j$-relative equivariant polynomial mappings is a module over
$\mathcal{P}(\Gamma)$ denoted by $\vecPsigmaj(\Gamma)$. The module $\vecP_{\sigma^0}(\Gamma)$ of the $\Gamma$-equivariant polynomial mappings on $V$ shall be denoted simply by
$\vecP(\Gamma)$. A  $\sigma^j$-relative equivariant mapping on $V$ may be regarded as
a $\Gamma$-equivariant mapping from $(\rho,V)$ to $(\sigma^j \eta,V).$ Therefore, the existence of a finite generating set
for $\vecPsigmaj(\Gamma)$ is guaranteed by Po\'enaru's Theorem (see \cite[Theorem XII 6.8]{GS69}).
For the particular case when $m=2$ and $\eta = \rho$, $\sigma \eta$ corresponds to the $\sigma$-dual representation  of $\rho$ that was studied in \cite{BM1}.

Once $\rho$ and $\eta$ in Definition~\ref{def:relative equivariant} denote the representations on the source and on the target of $g$, respectively, we simplify the notation and shall always write (\ref{eq:SIGMARELATIVEEQUIV}) from now on  as
\[ g(\gamma x) = \sigma^j(\gamma) \gamma g(x), \ \ \forall \gamma \in \Gamma, \ x \in V.\]

The next lemma is immediate from the definitions.

\begin{lemma} \label{lem:CARACTERIZATION}
For each $j \in \J,$ we have
\[ \vecPsigmaj(\Gamma)~=~\bigl\{g \in \vecP(K):
 g(\delta x) = \sigma^j(\delta) \delta g(x) \,,\;\forall\: x \in V\bigr\}.\]
\end{lemma}

Now, we introduce the \emph{$\sigma^j$-relative Reynolds operators} on
$\vecP(K)$ as the projections
$\vec{R}_j\!:\vecP(K) \to
\vecP(K)$,
\[
 \vec{R}_j(g)(x)~=~\frac{1}{m}
 \sum_{\;\gamma K} \overline{\sigma^j(\gamma)} \gamma g(\gamma x),\] where bar denotes  complex conjugation. We can rewrite each $\vec{R}_j$ as
\begin{equation}
 \begin{split} \label{eq:REYNOLDS2}
 \vec{R}_j(g)(x)~
 & = \frac{1}{m}
 \sum_{k=0}^{m-1} \overline{\sigma^{jk}(\delta)} \delta^k g(\delta^k x).
 \end{split}
\end{equation}

For the next proposition we recall that, since $\delta^m \in K, $
\begin{equation} \label{eq:ROOTUNITY}
\sum_{i = 0}^{m-1}\sigma(\delta)^{ik}= 1 + \sigma(\delta)^k + \sigma(\delta)^{2k} + \ldots + \sigma(\delta)^{(m-1)k} = 0,
\end{equation}
for all $k \in \{1, \ldots, m-1\}.$ If $I_{\vecP(K)}$ denotes the identity map on $\vecP(K)$,   we have:

\begin{proposition} \label{prop:PROPERTIES1}
For $j \in \J,$ the Reynolds operators $\vec{R}_j$'s satisfy the
following pro\-perties:
\begin{enumerate}[(i)]
 \item They are homomorphisms of $\mathcal{P}(\Gamma)$-modules and
 $       \vec{R}_0 + \vec{R}_1 + \ldots + \vec{R}_{m-1}~=~I_{\vecP(K)}. $
 \item They are idempotent projections, with $\IM(\vec{R}_j)~=~\vecPsigmaj(\Gamma).$
 \item For all $1 \leq j \leq m-1,$ we have $\vecP_{\sigma^j}(\Gamma) \cap \bigl(\vecP(\Gamma) + \ldots + \vecP_{\sigma^{j-1}}(\Gamma)\bigr) = \{0\}.$
 \end{enumerate}
\end{proposition}

\begin{proof}
{ (i)} follows  by the linearity of the action of $\Gamma$ and from
(\ref{eq:REYNOLDS2}) and (\ref{eq:ROOTUNITY}). (ii) and {(iii)} are direct extensions of \cite[Proposition 2.3]{BM4} for relative equivariant mappings.  \hfill \eproof
\end{proof}

\begin{remark} \normalfont
For $m=2,$ $\vec{R}_0$ and $\vec{R}_1$ are  $\vec{R}^\Gamma_{\Gamma_+}$ and $\vec{S}^\Gamma_{\Gamma_+}$ defined in \cite{antoneli}, respectively.
\end{remark}

The next theorem, our first main result,  follows now from Lemma~\ref{lem:CARACTERIZATION} and Proposition~\ref{prop:PROPERTIES1}:

\begin{theorem} \label{thm:RELATIVEDECOMP}
The following direct sum decomposition of $\mathcal{P}(\Gamma)$-modules holds:
\begin{equation}
\label{eq:DECOMPOSITION4} \vecP(K)~=~\bigoplus_{j=0}^{m-1} \vecP_{\sigma^{j}}(\Gamma).
\end{equation}
\end{theorem}

\section{Molien formulae} \label{sec:MOLIEN}

In this section we establish Molien formulae for  the
Hilbert-Poincar\'e series of the modules $\Psigmaj(\Gamma)$
and $\vecPsigmaj(\Gamma),$  $j \in \{0, \ldots, m-1\}$. As mentioned in Section~\ref{sec:INTRO}, the derivation of these is straightforward from the result obtained in \cite[Theorem 2.12]{Gat96} for $\Gamma$-equivariant mappings as in (\ref{eq:EQUIVARIANT}). However, here  we discuss a surprising feature related to the dependency of the series with respect to the representation of $\Gamma$  on the target of the mappings.


\subsection{Hilbert-Poincar\'e series and character formulae}

As we shalll see below, the formulae involve the character of the representation in the target space of the mappings. We recall that the $\it{character}$ of the \ representation $\eta$ of $\Gamma$ on $V$  is the function $\chi: \Gamma \to \C$ given by
\[\chi(\gamma) = \mathrm{tr}(\eta(\gamma)),\]
the trace of the matrix $\eta(\gamma).$

 We consider now the natural gradings
\begin{equation} \label{eq:GRADINGS}
 \Psigmaj(\Gamma)~=~\bigoplus_{d=0}^{\infty}
 \Psigmaj^d(\Gamma)
 \qquad\text{and}\qquad
 \vecPsigmaj(\Gamma)~=~\bigoplus_{d=0}^{\infty}
 \vecPsigmaj^{d}(\Gamma),
\end{equation}
where $\Psigmaj^d(\Gamma)$ is the space of homogeneous
 $\sigma^j$-relative invariant polynomials of degree $d$ and $\vecPsigmaj^d(\Gamma)$ is the
space of $\sigma^j$-relative equivariant polynomial mappings  with homogeneous
components of degree $d$. From (\ref{eq:DECOMPOSITION3}) and (\ref{eq:DECOMPOSITION4}) we have the direct sum decompositions of vector spaces
\[
 \mathcal{P}^d(K)~=~ \bigoplus_{j=0}^{m-1} \mathcal{P}_{\sigma^{j}}^d(\Gamma)
\]
and
\[
 \vecP^d(K)~=~\bigoplus_{j=0}^{m-1} \vecP_{\sigma^{j}}^d(\Gamma)
\]
for every degree $d \in \N$.

The \emph{Hilbert-Poincar\'e series} for $\mathcal{P}_{\sigma^j}(\Gamma)$
and for $\vecPsigmaj(\Gamma)$ are defined as the  formal
power series
\[
 \tilde{\Phi}_j(t)~=~\sum_{d=0}^{\infty}
 \dim \mathcal{P}_{\sigma^j}^d(\Gamma)\, t^d
 \qquad\text{and}\qquad
 \tilde{\Psi}_j(t)~=~\sum_{d=0}^{\infty}
 \dim \vecPsigmaj^d(\Gamma)\, t^d~.
\]

From the general Molien Theorem presented in  \cite[Theorem
2.12]{Gat96},  a formula for each Hilbert-Poincar\'e series above in terms of the normalized
Haar integral over $\Gamma$ is obtained directly, given by
\begin{equation} \label{eq:MOLIEN}
 \tilde{\Phi}_j(t)~=~\int_{\Gamma}
 \frac{\sigma^j(\gamma^{-1})}{\det (I-t \rho(\gamma))} \dd{\gamma}
 \quad\text{and} \quad
 \tilde{\Psi}_j(t)~=~\int_{\Gamma}
 \frac{\sigma^j(\gamma^{-1})\chi(\gamma^{-1})}{\det (I-t \rho(\gamma))} \dd{\gamma}.
\end{equation}

\begin{remark} \normalfont  \label{rmk:MOLIEN}
We stress that the dependency of the second integral with respect to  the character refers to  the representation $\eta$. In many applications in  dynamical systems, the computation of all generators
for the module of symmetric vector fields does not depend on the choice of the coordinate systems if it is done appropriately. In fact, this has been widely used by many authors to simplify the computation of explicit generators (see, for example, the classical example of ${\bf D}_n$-equivariants on ${\bf C}$
given in \cite[Examples 4.1(c) and 5.4(c)]{GS69}). There are cases, however, for which this can be a complicated computational task, included here the cases when the $K$-invariants and $K$-equivariants are not known  or are still complicated to be computed. For those cases, generators are usually found degree by degree. This is  when the knowledge {\it a priori} of the Hilbert-Poincar\'e series is most useful. Now, whereas certain changes of coordinates have no effect on the generators, it may lead to a change in the Hilbert-Poincar\'e series, once this corresponds to a switch in the character of the  symmetry group representation.  We illustrate this fact with a simple example, given in Subsection~\ref{subsec:MOLIEN EXAMPLE}.
\end{remark}

We now use \cite[Theorem 2.8]{antoneli} to obtain the dimensions of the homogeneous components of
(\ref{eq:GRADINGS}) in terms of the character function and of the
 integral over $\Gamma$. Once  $\sigma^j$-relative invariants on $V$ can be regarded as $\Gamma$-equivariant mappings from $(\rho,V)$ to $(\sigma^j,\C)$,  if $S^d V$ denotes the space of symmetric $k$-tensors over V, analogously as for index-2 case  (see expression (2.13) in \cite{antoneli}), here we  get
\begin{equation} \label{eq:DIM1}
 \dim \mathcal{P}_{\sigma^j}^d(\Gamma)~=~\int_{\Gamma}
 \sigma^j(\gamma)\chi_{(d)}(\gamma) \dd{\gamma},
\end{equation}
\begin{equation} \label{eq:DIM2} \dim \vecPsigmaj^d(\Gamma)~=~\int_{\Gamma}
 \sigma^j(\gamma) \chi_{(d)}(\gamma)\chi(\gamma) \dd{\gamma}, \end{equation}
where $\chi_{(d)}$ is the  character afforded by the induced
 action of $\Gamma$ on $S^dV$.

There is a
well-known recursive formula to compute the character $\chi_{(d)}$ of $S^d V,$ whose proof can be found  in for example \cite[Section 4]{antoneli1}:
\[
 d\, \chi_{(d)}(\gamma)~=~\sum_{i=0}^{d-1}
 \chi(\gamma^{d-i})\chi_{(i)}(\gamma),
\]
where $\chi_{(0)}=1$. Now, the Fubini-type result given in \cite[Proposition I 5.16]{Brocker} gives the
Haar integral  as an iteration of integrals. Applying the result to the character formulae (\ref{eq:DIM1}) and (\ref{eq:DIM2}) we get the
following integral expressions for the dimensions
of the spaces of relative invariants and relative equivariants, respectively:

\[
 \dim \mathcal{P}_{\sigma^j}^d(\Gamma)~=~\frac{1}{m}
 \Bigl[\sum _{k = 0}^{m-1}
 \int_{K} \sigma^j(\delta^{k} \gamma)\chi_{(d)}(\delta^{k} \gamma)\dd{\gamma} \Bigr] \]
and
\[
 \dim \vecPsigmaj^d(\Gamma)~=~\frac{1}{m}
 \Bigl[\sum _{k = 0}^{m-1} \int_{K} \sigma^j(\delta^{k} \gamma)\chi_{(d)}(\delta^{k} \gamma)
 \chi(\delta^k \gamma)\dd{\gamma}\Bigr]~.
\]
In particular,
\[ \dim \mathcal{P}^d(\Gamma)~=~\frac{1}{m}
 \Bigl[\int_{K} \chi_{(d)}(\gamma) \dd{\gamma} +
 \int_{K}\chi_{(d)}(\delta \gamma)\dd{\gamma} + \ldots + \int_{K}\chi_{(d)}(\delta^{m-1} \gamma)\dd{\gamma}\Bigr]
\]
and \[
 \dim \vecP^d(\Gamma)~=~\frac{1}{m}
 \Bigl[\int_{K} \chi_{(d)}(\gamma)\chi(\gamma)\dd{\gamma} +
 \ldots + \int_{K}\chi_{(d)}(\delta^{m-1} \gamma)
 \chi(\delta^{m-1}\gamma)\dd{\gamma}\Bigr]~.
\]

From (\ref{eq:ROOTUNITY}), we obtain
\[
 \sum_{j=0}^{m-1} \dim \mathcal{P}_{\sigma^j}^d(\Gamma) = \int_{K}
 \chi_{(d)}(\gamma)\dd{\gamma}~=~\dim \mathcal{P}^d(K)
\] and
\[
 \sum_{j=0}^{m-1} \dim \vecP_{\sigma^j}^d(\Gamma) = \int_{K}
 \chi_{(d)}(\gamma)\chi(\gamma)\dd{\gamma}~=~\dim \vecP^d(K).
\]

\subsection{Example} \label{subsec:MOLIEN EXAMPLE}

We apply the Molien formulae (\ref{eq:MOLIEN}) for each
Hilbert-Poincar\'e series of relative invariant and equivariant for the action on ${\bf C}^2$ of the cartesian product  $\Gamma = \Zthree \times \Zthree$ of the order-3 cyclic group ${\bf Z}_3 =
\langle e^{2 \pi i/3} \rangle$.

In standard coordinates $z=(z_1,z_2) \in {\bf C}^2,$ we consider the diagonal action
\[ \xi z~=~(\xi_1 z_1,\xi_2 z_2), \]
$\xi = (\xi_1,\xi_2) \in \Gamma,$ and the epimorphism $\sigma: \Gamma \rightarrow \Zthree$ such that $K = \ker \sigma = \Zthree \times \{\mathrm{I} \}$. Let us take $\delta = (1,e^{\frac{2\pi i}{3}}) \in \Gamma \setminus K$. \\

On the target space of the mappings we consider the standard representation given by complex product on each component.  Hence,
\begin{equation} \label{eq:CHARACTER}
 \chi(\xi_1, \xi_2) = \xi_1 + \xi_2.
 \end{equation}
On  each component of the domain, we identify $ z_i \equiv (z_i, \bar{z_i}),$ $i=1,2,$
and so the representation of $\Gamma$ on the domain is generated by the order-4 matrices
\begin{equation} \label{eq:MATRICES} \left( \begin{array}{cc|cc}
e^{i 2 \pi/3}  & 0 &  & 0  \\
0&   e^{-i 2 \pi/3} & &   \\  \hline
 & 0& & {\rm I} \\
\end{array} \right), \ \ \
\left( \begin{array}{cc|cc}
{\rm I}  &  &  & 0  \\ \hline
&  &  e^{i 2 \pi/3} &    \\
0 & & & e^{-i 2 \pi/3} \\
\end{array} \right).
\end{equation}
For $j=0$, $\sigma^j(\gamma^{-1}) = 1, \forall \gamma \in \Gamma$. We now use (\ref{eq:MOLIEN}) to compute
\begin{align} \label{eq:SERIES1}
\tilde{\Phi}_0(t) \ = \ &  \int_{\Gamma}
 \frac{1}{\det (I-t \rho(\gamma))} \dd{\gamma} =
 \frac{1}{9} \sum_{\gamma \in \Gamma}  \frac{1}{\det (I-t \rho(\gamma))}  = \nonumber \\[1mm]
 = & \ \  1 + 2 t^2+ 4 t^3 + 3 t^4+ 8 t^5 + 12 t^6 + \cdots. \\[1mm] \nonumber
 \end{align}
For $j=1$, we have $\sigma(\xi_1, \xi_2) = \xi_2$, and  (\ref{eq:MOLIEN}) gives
\begin{align} \label{eq:SERIES2}
\tilde{\Phi}_1(t) \ = \ &  \int_{\Gamma}
 \frac{\sigma(\gamma^{-1})}{\det (I-t \rho(\gamma))} \dd{\gamma} =
 \frac{1}{9} \sum_{\gamma \in \Gamma}  \frac{\sigma(\gamma^{-1})}{\det (I-t \rho(\gamma))}  = \nonumber\\[1mm]
 = & \ \  t+ t^2 + 2 t^3 + 5 t^4 + 6 t^5 + 9 t^6 + \cdots. \\[1mm] \nonumber
 \end{align}
For $j=2$, we have $\sigma^2(\xi_1, \xi_2) = \xi_2^2$, and  (\ref{eq:MOLIEN}) gives
\begin{align} \label{eq:SERIES3}
\tilde{\Phi}_2(t) \ = \ &  \int_{\Gamma}
 \frac{\sigma^2(\gamma^{-1})}{\det (I-t \rho(\gamma))} \dd{\gamma} =
 \frac{1}{9} \sum_{\gamma \in \Gamma}  \frac{\sigma^2(\gamma^{-1})}{\det (I-t \rho(\gamma))}  =
  \nonumber\\[1mm]
 = & \ \ \tilde{\Phi}_1(t) \  =  \ t+ t^2 + 2 t^3 + 5 t^4 + 6 t^5 + 9 t^6 + \cdots. \\[1mm] \nonumber
 \end{align}
 Also, from  (\ref{eq:MOLIEN}) and (\ref{eq:CHARACTER}) we find
\begin{align} \label{eq:SERIES4}
\tilde{\Psi}_0(t) \ = \ &  \int_{\Gamma}
 \frac{\chi(\gamma^{-1})}{\det (I-t \rho(\gamma))} \dd{\gamma} =
 \frac{1}{9} \sum_{\gamma \in \Gamma}  \frac{\overline{\chi(\gamma)}}{\det (I-t \rho(\gamma))}  =
 \nonumber \\[1mm]
 = & \ \  2t+ 2t^2+4t^3+10t^4+12t^5+18 t^6 + \cdots, \\ \nonumber
 \end{align}
\begin{align} \label{eq:SERIES5}
\tilde{\Psi}_1(t) \ = \ &  \int_{\Gamma}
 \frac{\sigma(\gamma^{-1})  \chi(\gamma^{-1})}{\det (I-t \rho(\gamma))} \dd{\gamma} =
 \frac{1}{9} \sum_{\gamma \in \Gamma}  \frac{\sigma(\gamma^{-1}) \overline{\chi(\gamma)}}{\det (I-t \rho(\gamma))}  =
 \nonumber \\[1mm]
 = & \ \ t + 2 t^2+ 4 t^3+ 8 t^4+ 12 t^5+ 18 t^6 + \cdots, \\\nonumber
 \end{align}
and
\begin{align} \label{eq:SERIES6}
\tilde{\Psi}_2(t) \ = \ &  \int_{\Gamma}
 \frac{\sigma^2(\gamma^{-1})  \chi(\gamma^{-1})}{\det (I-t \rho(\gamma))} \dd{\gamma} =
 \frac{1}{9} \sum_{\gamma \in \Gamma}  \frac{\sigma^2(\gamma^{-1}) \overline{\chi(\gamma)}}{\det (I-t \rho(\gamma))}  =
 \nonumber \\[1mm]
 = & \ \ 1 + 3 t^2 + 6 t^3+ 6 t^4 + 14 t^5 + 21 t^6 \cdots. \\\nonumber
 \end{align}
If we considered in the target space the representation given by
(\ref{eq:MATRICES}), then it is easy to see that the change in the character would affect  the Hilbert-Poincar\'e series for the relative equivariants to be twice the expressions given in (\ref{eq:SERIES4}), (\ref{eq:SERIES5}) and (\ref{eq:SERIES6}).  Hence, this simple example illustrates how these integrals are not invariant by changes of coordinates that on the other hand leave invariant the set of generators. In Subsection~\ref{example:generators} we shall come back to this example, comparing the series obtained here with the dimensions of homogeneous vector spaces given by direct computation.

\section{Generators of relative equivariants}

\label{sec:GERADORES}

In this section, for $0 \leq j \leq m-1$,  we compute generators for $\vecPsigmaj(\Gamma)$ as modules over $\mathcal{P}(\Gamma),$  from a Hilbert basis of $\mathcal{P}(K)$ and a set of generators  of $\vecP(K),$ where $K$ is the index-$m$ normal subgroup formed only by the symmetries of $\Gamma$. The key idea is  to use the decomposition given by  Theorem \ref{thm:RELATIVEDECOMP} to transfer generating sets from one module to another. We finish the section by applying the result for one example.

\subsection{Generators of relative invariants}

We see from  \cite[Theorems 3.2 and 3.5]{BM4} how to compute generators for $\mathcal{P}(\Gamma),$ as a ring, from a Hilbert basis of $\mathcal{P}(K),$ for $m=2$ and $m > 2$ respectively. The next theorem extends this, giving generators of the  $\sigma^j$-relative invariants for any $j = 1, \ldots m-1,$ with $m \geq 2.$

\begin{theorem} \label{thm:RELATIVEGENERATORS}
Let $\{u_1,\ldots,u_s\}$ be a
Hilbert basis of $\mathcal{P}(K).$ Then, for each $j \in \Jo,$ the $\sigma^j$-relative invariants
\[R_1(u_1)^{\alpha_{11}} \ldots R_1(u_s)^{\alpha_{1s}} \ldots R_{m-1}(u_1)^{\alpha_{m-1, 1}} \ldots R_{m-1}(u_s)^{\alpha_{m-1,s}},\]
\noindent such that $\displaystyle \sum_{l=1}^{m-1} l (\alpha_{l1} + \ldots + \alpha_{ls}) \equiv j \ \text{mod} \ m,$ form a generating set of the module $\Psigmaj(\Gamma)$ over the ring $\mathcal{P}(\Gamma).$

\end{theorem}

\begin{proof}
Let $\{u_1,\dots,u_s\}$ be a
Hilbert basis of $\mathcal{P}(K)$. For each $j \in \J,$ consider the $\sigma^j$-relative Reynolds operators on $\mathcal{P}(K)$ given by
\[R_j(f)(x) = \frac{1}{m} \sum_{k=0}^{m-1} \overline{\sigma^{jk}(\delta)} f(\delta^k x).\] It follows from \cite[Proposition 2.3]{BM4} that if $f \in \Psigmaj(\Gamma),$ then $R_j(f) = f.$ Hence,
\begin{equation}
\label{eq:Rj}
f(x) = \frac{1}{m-1} \sum_{k=1}^{m-1}\overline{\sigma^{j}(\delta^{k})}f(\delta^k x), \quad \forall \ x \in V.
\end{equation}
Also from \cite[Proposition 2.3]{BM4}, for each $i \in \{1, \ldots, s\}$ we have
$u_i = \displaystyle \sum_{j=0}^{m-1}R_j(u_i).$ So \[\{R_0(u_i), R_1(u_i), \ldots, R_{m-1}(u_i): 1 \leq i \leq s\}\] is a new Hilbert basis of $\mathcal{P}(K).$ Using multi-index notation, every polynomial function $f \in \Psigmaj(\Gamma)$ can be written as
\[f~=~\sum a_{\alpha} R_0(u_i)^{\alpha_{\!0}} \,
 R_{1}(u_i)^{\alpha_{\!1}} \ldots R_{m-1}(u_i)^{\alpha_{m-1}},\]
 with $a_{\alpha} \in \C,$ where $\alpha = (\alpha_{\!0},\ldots,\alpha_{m-1}),$ $\alpha_j = (\alpha_{j\!1},\ldots,\alpha_{js}) \in \N^s$ and $R_j(u_i)^{\alpha_{\!j}} = R_{j}(u_1)^{\alpha_{j\!1}} \ldots R_{j}(u_s)^{\alpha_{js}},$ for $j  =0, \ldots, m-1.$ Now, for every $k \in \N$ and $x \in V,$ we have \[R_j(u_i)(\delta^k x) = \sigma^j(\delta^k)R_j(u_i)(x) = \sigma(\delta)^{jk}R_j(u_i)(x),\] from which we obtain
\begin{equation}
\label{eq:FDELTA}
 f(\delta^k x)~=~\sum \sigma(\delta)^{k N} a_{\alpha} \,R_0(u_i)^{\alpha_{\!0}} \,
 R_{1}(u_i)^{\alpha_{\!1}} \ldots R_{m-1}(u_i)^{\alpha_{m-1}}(x),
\end{equation} where $N = \displaystyle \sum_{l=1}^{m-1} l (\alpha_{l1} + \ldots + \alpha_{ls}) \in \N.$ We now use (\ref{eq:FDELTA}) in (\ref{eq:Rj}) to get
\[
 \sigma(\delta)^{N - j} + \sigma(\delta)^{2(N - j)} + \ldots + \sigma(\delta)^{(m-1)(N - j)} = m-1.\] It follows from
 (\ref{eq:ROOTUNITY}) that $\sigma(\delta)^{N - j} = 1,$ that is, $N - j \equiv$ 0 mod $m.$ Therefore, the generators of $\Psigmaj(\Gamma),$ as a module over $\mathcal{P}(\Gamma),$ are given by the products
 \[R_1(u_1)^{\alpha_{11}} \ldots R_1(u_s)^{\alpha_{1s}} \ldots R_{m-1}(u_1)^{\alpha_{m-1,1}} \ldots R_{m-1}(u_s)^{\alpha_{m-1,s}},
\] for $ \displaystyle \sum_{l=1}^{m-1} l (\alpha_{l1} + \ldots + \alpha_{ls}) \equiv j$ mod $m.$ \hfill \eproof
\end{proof}

\quad

The next result is direct from  (\ref{eq:DECOMPOSITION3}). It gives  generators of $\mathcal{P}(K)$ as a module over  $\mathcal{P}(\Gamma)$, which is a  crucial tool  for Theorem~\ref{thm:MAIN2}.

\begin{corollary} \label{cor:RINGENERATORS}
For each $j = 1, \ldots, m-1$, let $\mathcal{B}_j$ be a generating set of the module $\Psigmaj(\Gamma)$ over $\mathcal{P}(\Gamma)$ given by Theorem~\ref{thm:RELATIVEGENERATORS}.  Set $\mathcal{B}_0 = \{1\}.$ Then $\mathcal{B} =  \bigcup_{j = 0}^{m-1}\mathcal{B}_j$ is a generating set of
the module $\mathcal{P}(K)$ over
$\mathcal{P}(\Gamma)$.
\end{corollary}

\subsection{Computation of equivariants and relative equivariants}

By the $K$-invariant theory, we obtain generating sets for each $\vecP(\Gamma)$ and
$\vecPsigmaj(\Gamma)$, $j=1, \ldots, m-1$, as  $\mathcal{P}(\Gamma)$-modules.

\begin{lemma} \label{lem:MAIN2}
Let  $\mathcal{B} = \{v_0 \equiv 1, v_1, \ldots, v_p\}$ be the
generating set of the module $\mathcal{P}(K)$ over the
ring $\mathcal{P}(\Gamma)$ given as
in Corollary \ref{cor:RINGENERATORS}. Let $\{H_0,\ldots,H_q\}$ be a
generating set of the module $\vecP(K)$ over the ring
$\mathcal{P}(K)$. Then
\[
 \{H_{ik}~=~v_i H_k:i=0,\ldots,p \,;\; k=0,\ldots,q\}
\]
is a generating set of  $\vecP(K)$  as a $\mathcal{P}(\Gamma)$-module.
\end{lemma}

\begin{proof}
Let $G \in\vecP(K)$. Then
\begin{equation} \label{forma1}
 G~=~\sum_{k=0}^{q} p_k \, H_k,
\end{equation}
with $p_k \in\mathcal{P}(K)$. Since
$\mathcal{B}$ generates  $\mathcal{P}(K)$ as a $\mathcal{P}(\Gamma)$-module, it follows that
\begin{equation} \label{forma2}
 p_k=~\sum_{i=0}^{p} p_{ik} \, v_i,
\end{equation}
with $p_{ik}\in\mathcal{P}(\Gamma)$. From \eqref{forma1} and
\eqref{forma2} we get
\[
 G~=~\sum_{k=0}^{q} \left( \sum_{i=0}^{p}
     p_{ik} \, v_i \right)
     H_k~=~\sum_{i,k=0}^{p,q}
     p_{ik} \, \big(v_i H_k
     \big).
\]
\hfill \eproof
\end{proof}

\begin{lemma}
Let $\{H_{00},\dots,H_{pq}\}$ be
a generating set of
$\vecP(K)$ over the ring $\mathcal{P}(\Gamma)$ given
as in Lemma \ref{lem:MAIN2}. Then, for each $j \in \J,$ the set
\[
 \{\vec{R}_j(H_{ik}):
 i=0,\ldots,p \,;\; k=0,\ldots,q\}
\]
generates $\vecPsigmaj(\Gamma)$ as a  $\mathcal{P}(\Gamma)$-module.

\end{lemma}
\begin{proof}
Let $\tilde{G} \in \vecPsigmaj(\Gamma)$. From
Proposition~\ref{prop:PROPERTIES1}, there exists $G
\in \vecP(K)$ such that
$\tilde{G}~=\vec{R}_j(G)$. As $\{H_{00},\ldots,H_{pq}\}$ is a generating set of the
$\mathcal{P}(\Gamma)$-module $\vecP(K)$, we write
\[
 G~=~\sum_{i,k=0}^{p,q}
     p_{ik} \, H_{ik}
\]
with $p_{ik}\in\mathcal{P}(\Gamma)$. Therefore
\[
 \tilde{G}~=~\vec{R}_j\bigl(\sum_{i,k=0}^{p,q}
     p_{ik} \, H_{ik}\bigr)~=~\sum_{i,k=0}^{p,q}
     p_{ik} \,
 \vec{R}_j(H_{ik}).
\]
\hfill \eproof
\end{proof}

It is now  immediate from the two lemmas above the following
result:

\begin{theorem} \label{thm:MAIN2}
Let $\mathcal{B} = \{v_0 \equiv 1, v_1, \ldots, v_p\}$ be the
generating set of the module $\mathcal{P}(K)$ over  $\mathcal{P}(\Gamma)$ given by
Corollary \ref{cor:RINGENERATORS}. Let $\{H_0,\ldots,H_q\}$  a
generating set of the module $\vecP(K)$ over the ring
$\mathcal{P}(K)$. Then
\[
 \{\vec{R}_j(v_iH_k):
   i=0,\ldots,p \,;\; k=0,\ldots,q\}
\]
is a set of generators for   $\vecPsigmaj(\Gamma)$ over
$\mathcal{P}(\Gamma)$.
\end{theorem}

We finally remark that Theorem~\ref{thm:MAIN2} is a very useful result even in an equivariant context. In fact,  it provides an algorithmic way to compute generators of $\Gamma$-equivariants ($j=0$) from the knowledge of equivariants only by a subgroup, as long as $\Gamma$ admits a normal subgroup
of finite index. This corresponds to the equivariant version of \cite[Theorem 3.5]{BM4}, which provides a Hilbert basis for the $\Gamma$-invariants from the knowledge of invariants by a smaller subgroup.

\subsection{Example} \label{example:generators}

We illustrate an application of Theorems \ref{thm:RELATIVEGENERATORS}, \ref{thm:MAIN2} for the example of Subsection~\ref{subsec:MOLIEN EXAMPLE} and present Tables 1 and 2 to compare the results with the Hilbert-Poincar\'e series computed in that subsection. The polynomials
\[u_1(z_1,z_2)=z_1\bar{z}_1, \quad u_2(z_1,z_2)=z_1^3, \quad u_3(z_1,z_2)=\bar{z}_1^3,\] \[u_4(z_1,z_2)=z_2, \quad u_5(z_1,z_2)=\bar{z}_2\]
form a Hilbert basis of $\mathcal{P}(K).$ In this case, we have
$R_0(u_i) = u_i,$ for $\ 1 \leq i \leq 3,$ $R_1(u_4) = u_4$ and $R_2(u_5) = u_5.$ All other projections are zero. From \cite[Example 3.6]{BM4}, a Hilbert basis of the ring $\mathcal{P}(\Gamma)$ is  $\{z_1\bar{z}_1, \ z_1^3, \ \bar{z}_1^3, \ z_2\bar{z}_2, \ z_2^3, \ \bar{z}_2^3\}.$ By Theorem \ref{thm:RELATIVEGENERATORS}, the module $\mathcal{P}_{\sigma}(\Gamma)$ is generated by
 \[R_1(u_i),\ R_2(u_i)R_2(u_j), \ \  1 \leq i, j \leq 5, \]
 and  $\mathcal{P}_{\sigma^2}(\Gamma)$  by
 \[R_2(u_i),\ R_1(u_i)R_1(u_j), \ \  1 \leq i, j \leq 5, \]
 both over the ring $\mathcal{P}(\Gamma).$ It follows that  the sets $\{z_2, \ \bar{z}_2^2\}$ and $\{\bar{z}_2, \ z_2^2\}$ generate $\mathcal{P}_{\sigma}(\Gamma)$ and $\mathcal{P}_{\sigma^2}(\Gamma),$ respectively.

From these, we obtain the data given in Table~1, in agreement with
(\ref{eq:SERIES1}), (\ref{eq:SERIES2}) and (\ref{eq:SERIES3}).

\begin{table}[h] \label{table:DIMENSIONS INVARIANTS}
\begin{center}
 \begin{tabular}{|c|c|c|c|} \hline
 & $\dim {\cal P}^d(\Gamma)$ & $\dim {{\cal P}^d}_{\sigma}(\Gamma)$ &
 $\dim {{\cal P}}^d_{\sigma^2}(\Gamma)$ \\ \hline \hline
$d=0$ & 1  & 0& 0 \\ \hline
$d=1$ & 0 &1 & 1\\ \hline
$d=2$ & 2 & 1 & 1\\ \hline
$d=3$ & 4 & 2 & 2\\ \hline
$d=4$ & 3 &5 & 5\\ \hline
$d=5$ & 8 &6 & 6\\ \hline
$d=6$ & 12 & 9& 9 \\ \hline
  \end{tabular}
\end{center}
\caption{Dimensions of the vector spaces of homogeneous $\sigma^j$-relative invariants of degree $d$ for $d=0, \ldots, 6.$}
\end{table}

 By Corollary \ref{cor:RINGENERATORS}, $\mathcal{B} = \{v_0, \ldots, v_4\}$  is a generating set of $\mathcal{P}(K)$ as a module over $\mathcal{P}(\Gamma),$
where \[v_0(z_1,z_2)= 1, \quad v_1(z_1,z_2)=z_2, \quad v_2(z_1,z_2)=\bar{z}_2,\] \[v_3(z_1,z_2)=z_2^2, \quad v_4(z_1,z_2)=\bar{z}_2^2.\]
 Straightforward calculations give
\[H_0(z_1,z_2)= (z_1,0) \quad H_1(z_1,z_2)=(\bar{z}_1^2,0), \quad H_2(z_1,z_2)=(0,1)\]
as generators of $\vecP(K)$ over $\mathcal{P}(K).$ Lemma \ref{lem:MAIN2} provides
\begin{align*}
 H_{00}(z_1,z_2)~ & =~H_0(z_1,z_2)~=~(z_1,0) \,, \\[1mm]
 H_{01}(z_1,z_2)~ & =~H_1(z_1,z_2)~=~(\bar{z}_1^2,0) \,, \\[1mm]
 H_{02}(z_1,z_2)~ & =~H_2(z_1,z_2)~=~(0,1) \,, \\[1mm]
 H_{10}(z_1,z_2)~ & = v_1(z_1,z_2)H_0(z_1,z_2) ~=~(z_1z_2,0)\,, \\[1mm]
 H_{11}(z_1,z_2)~ & =~v_1(z_1,z_2)H_1(z_1,z_2) ~=~(\bar{z}_1^2z_2,0) \,, \\[1mm]
 H_{12}(z_1,z_2)~ & =~v_1(z_1,z_2)H_2(z_1,z_2) ~=~(0,z_2) \,, \\[1mm]
 H_{20}(z_1,z_2)~ & =~v_2(z_1,z_2)H_0(z_1,z_2) ~=~(z_1\bar{z}_2,0) \,, \\[1mm]
 H_{21}(z_1,z_2)~ & =~v_2(z_1,z_2)H_1(z_1,z_2) ~=~(\bar{z}_1^2\bar{z}_2,0)\,, \\[1mm]
 H_{22}(z_1,z_2)~ & =~v_2(z_1,z_2)H_2(z_1,z_2) ~=~(0,\bar{z}_2)   \,, \\[1mm]
 H_{30}(z_1,z_2)~ & =~v_3(z_1,z_2)H_0(z_1,z_2) ~=~(z_1z_2^2,0) \,, \\[1mm]
 H_{31}(z_1,z_2)~ & =~v_3(z_1,z_2)H_1(z_1,z_2) ~=~(\bar{z}_1^2z_2^2,0)\,, \\[1mm]
 H_{32}(z_1,z_2)~ & =~v_3(z_1,z_2)H_2(z_1,z_2) ~=~(0,z_2^2)   \,, \\[1mm]
 H_{40}(z_1,z_2)~ & =~v_4(z_1,z_2)H_0(z_1,z_2) ~=~(z_1\bar{z}_2^2,0) \,, \\[1mm]
 H_{41}(z_1,z_2)~ & =~v_4(z_1,z_2)H_1(z_1,z_2) ~=~(\bar{z}_1^2\bar{z}_2^2,0)\,, \\[1mm]
 H_{42}(z_1,z_2)~ & =~v_4(z_1,z_2)H_2(z_1,z_2) ~=~(0,\bar{z}_2^2)
\end{align*}
as generators of $\vecP(K)$ over $\mathcal{P}(\Gamma).$ From Theorem~\ref{thm:MAIN2}  we compute $\vec{R}_0(H_{i2}) = H_{i2},$ $\vec{R}_0(H_{0j}) = H_{0j},$ $\vec{R}_1(H_{ij}) = H_{ij},$ $\vec{R}_1(H_{l2}) = H_{l2},$ $\vec{R}_2(H_{lj}) = H_{lj},$ $\vec{R}_2(H_{02}) = H_{02},$ for $i= 1,4,$ $j=0,1$ and $l=2,3.$ All the other projections are zero. Hence, the sets
\[\bigl \{(z_1,0), \ (\bar{z}_1^2,0), \ (0,z_2), \ (0,\bar{z}_2^2)\bigr \},\]
\[\bigl \{(z_1z_2,0), \ (\bar{z}_1^2z_2,0), \ (0,\bar{z}_2), \ (0,z_2^2), \ (z_1\bar{z}_2^2,0), \ (\bar{z}_1^2\bar{z}_2^2,0)\bigr \}\] and
\[\bigl \{(0,1), \ (z_1\bar{z}_2,0), \ (\bar{z}_1^2\bar{z}_2,0), \ (z_1z_2^2,0), \ (\bar{z}_1^2z_2^2,0)\bigr \}\] generate $\vecP(\Gamma),$ $\vec{P}_{\sigma}(\Gamma)$ and $\vec{P}_{\sigma^2}(\Gamma)$ over $\mathcal{P}(\Gamma),$ respectively.

Therefore, we obtain the general form of a $\Gamma$-equivariant polynomial mapping
$$g(z_1,z_2) = \bigl(f_1(z) z_1 + f_2(z) \bar{z}_1^2, f_3(z) z_2 + f_4(z) \bar{z}_2^2 \bigr),$$
of a $\sigma$-relative equivariant
$$h(z_1,z_2) = \bigl(f_5(z) z_1z_2 + f_6(z) \bar{z}_1^2z_2 + f_7(z) z_1 \bar{z}_2^2 + f_8(z) \bar{z}_1^2 \bar{z}_2^2, f_9(z) \bar{z}_2 + f_{10}(z)z_2^2 \bigr)$$
and of a $\sigma^2$-relative equivariant  $$p(z_1,z_2) = \bigl(f_{11}(z)z_1 \bar{z}_2 + f_{12}(z) \bar{z}_1^2 \bar{z}_2 + f_{13}(z) z_1 z_2^2 + f_{14}(z) \bar{z}_1^2z_2^2, f_{15}(z)\bigr),$$ where $z = (z_1,z_2) \in {\bf C}^2$ and $f_i \in \mathcal{P}(\Gamma)$ for all $i = 1, \ldots, 15.$

Finally we also obtain, from the list of generators above, the data given in Table~2, which is in agreement with (\ref{eq:SERIES4}), (\ref{eq:SERIES5}) and (\ref{eq:SERIES6}).

\begin{table}[h] \label{table:DIMENSIONS EQUIVARIANTS}
\begin{center}
 \begin{tabular}{|c|c|c|c|} \hline
 & $\dim {\vec{\cal P}}^d(\Gamma)$ & $\dim {{\vec{\cal P}}^d}_{\sigma}(\Gamma)$ &
 $\dim {{\vec{\cal P}}}^d_{\sigma^2}(\Gamma)$ \\ \hline \hline
$d=0$ & 0 & 0& 1 \\ \hline
$d=1$ & 2 &1 & 0\\ \hline
$d=2$ & 2 & 2 & 3 \\ \hline
$d=3$ & 4 & 4& 6\\ \hline
$d=4$ & 10 &8 & 6\\ \hline
$d=5$ & 12 &12& 14 \\ \hline
$d=6$ & 18 & 18 &  21\\ \hline
  \end{tabular}
\end{center}
\caption{Dimensions of the vector spaces of homogeneous $\sigma^j$-relative equivariants of degree $d$,  for $d=0, \ldots, 6.$}
\end{table}

\vspace*{1cm}

\noindent {\it Acknowledgments.} The research of M.M. was supported by FAPESP,  BPE grant 2013/11108-7.



\begin{thebibliography}{99}

\bibitem{antoneli} F.~Antoneli, P.H.~Baptistelli, A.P.S.~Dias, M.~Manoel,
Invariant theory and reversible-equivariant vector fields, {\it Journal of Pure and Applied Algebra} {\bf 213} (2009) 649--663.

\bibitem{antoneli1} F.~Antoneli, A.P.S.~Dias, P.C.~Matthews,
Invariants, Equivariants and Characters in Symme\-tric Bifurcation
Theory. \textit{Proc. Roy. Soc. Edinburgh} {\bf 137A} (2007) 01--36.

\bibitem{BM4} P.H.~Baptistelli, M.~Manoel, Invariants and relative invariants under compact Lie groups, {\it Journal of Pure and Applied Algebra} {\bf 217} (2013) 2213--2220.

\bibitem{BM3} P.H.~Baptistelli, M.~Manoel, The $\sigma$-isotypic decomposition and the $\sigma$-index of reversible-equivariant systems. {\it Topology and its Applications} {\bf 159} (2012) 389--396.

\bibitem{BM1} P.H.~Baptistelli, M.~Manoel, The
classification of reversible-equivariant steady-state bifurcations
on self-dual spaces, {\it Math. Proc. Cambridge Philos. Soc.} {\bf 145} (2)
(2008) 379--401.

\bibitem{Buono} P.~L.~Buono, J.~S.~W.~Lamb, R.~M.~Roberts, Bifurcation and Branching of Equilibria of Reversible Equivariant Vector Fields. {\it Nonlinearity} {\bf 21} (2008), 625--660.

\bibitem{Brocker} T.~Br\"{o}cker, T.~tom Dieck,
\textit{Representations of Compact Lie Groups}.
Graduate Texts in Mathematics {\bf 98},
Springer-Verlag, New York, 1995.

\bibitem{chossat} P.~Chossat, R.~Lauterbach, {\it Methods in equivariant bifurcations and dynamical systems }. Advanced series in nonlinear dynamics {\bf 15}, World Scientific, 2000.

\bibitem{Naschie} M.S. El Naschie,  On the nature of complex time, diffusion and the two-slit experiment. {\it Chaos, Solitons and Fractals} {\bf 5}(6) (1995) 1031--1032.



\bibitem{Gat} K.~Gatermann,
\textit{Computer Algebra Methods for Equivariant Dynamical Systems}.
Lecture Notes in Mathematics {\bf 1728}, Springer-Verlag,
Berlin-Heidelberg, 2000.

\bibitem{Gat96} K.~Gatermann,
Semi-invariants, Equivariants and Algorithms.
\textit{Appl. Algebra Eng. Commun. Comput.} \textbf{7} (1996) 105--124.


\bibitem{GS69} M.~Golubitsky, I.~Stewart, D.~Schaeffer,
\textit{Singularities e Groups in Bifurcation Theory, Vol. II}.
Appl. Math. Sci. {\bf 69}, Springer-Verlag, New York, 1985.

\bibitem{Lamb99} J. S. W. Lamb, R. M. Roberts, Reversible equivariant linear systems. {\it J. Diff. Eq.} {\bf 159} (1999) 239 - 279.

\bibitem{MSSF} A. Mejias, L. Di.G. Sigalotti, E. Sira, F. Felice, On El Naschie's complex time, Hawking's imaginary time and special relativity. {\it Chaos, Solitons and Fractals} {\bf 19} (2004) 773--777.



\bibitem{SGP} I. Stewart, M. Golubitsky, M. Pivato,  Symmetry groupoids and patterns of synchrony in coupled cell networks. {\it SIAM J. Appl. Dyn. Syst.} {\bf 2}(4) (2003) 609--646.


\bibitem{sturm} B.~Sturmfels, \textit{Algorithms in Invariant Theory}.
Springer-Verlag, New York, 1993.

\end{thebibliography}
\end{document}